\documentclass[12pt]{article}
\usepackage[utf8]{inputenc}
\usepackage{amsmath, amssymb, amsthm}
\usepackage{enumitem}
\usepackage{tikz}
\usetikzlibrary{positioning}
\usepackage{hyperref}
\usepackage{geometry}
\usepackage{graphicx}
\usepackage[T2A]{fontenc}
\usepackage[utf8]{inputenc}
\usepackage[russian,english]{babel}
\usepackage{authblk} 
\newtheorem{theorem}{Theorem}[section]
\newtheorem{lemma}[theorem]{Lemma}

\newtheorem{corollary}[theorem]{Corollary}

\newtheorem{remark}[theorem]{Remark}
\newtheorem{conjecture}[theorem]{Conjecture} 

\DeclareMathOperator{\GW}{GW}
\newcommand{\prob}{\mathbb{P}}  
\DeclareMathOperator{\E}{\mathbb{E}}
\DeclareMathOperator{\Var}{Var}
\DeclareMathOperator{\Cov}{Cov}

\newcommand{\ltr}{\textsc{ltr}}
\newcommand{\rtl}{\textsc{rtl}}

\title{Greedy Routing in a Sequentially Grown One-Dimensional Random Graph}
\author{Alexander Ponomarenko}
\affil{Laboratory of Algorithms and Technologies for Network Analysis, HSE University, Nizhny Novgorod, Russia}
\date{\today}

\begin{document}

\maketitle

\begin{abstract}
We analyze greedy routing in a random graph $G_n$ constructed on the vertex set $V = \{1, 2, \dots, n\}$ embedded in $\mathbb{Z}$.
Vertices are inserted according to a uniform random permutation $\pi$, and each newly inserted vertex connects to its nearest already-inserted neighbors on the left and right (if they exist).
This work addresses a conjecture originating from empirical studies~\cite{Ponomarenko2011, malkov2012scalable, пономаренко2012структура}, which observed through simulations that greedy search in sequentially grown graphs exhibits logarithmic routing complexity across various dimensions.
While the original claim was based on experiments and geometric intuition, a rigorous mathematical foundation remained open.
Here, we formalize and resolve this conjecture for the one-dimensional case.
For a greedy walk $\GW$ starting at vertex $1$ targeting vertex $n$---which at each step moves to the neighbor closest to $n$---we prove that the number of steps $S_n$ required to reach $n$ satisfies $S_n = \Theta(\log n)$ with high probability.
Precisely, $S_n = L_n + R_n - 2$, where $L_n$ and $R_n$ are the numbers of left-to-right and right-to-left minima in the insertion-time permutation.
Consequently, $\E[S_n] = 2H_n - 2 \sim 2\log n$ and $\prob(S_n \geq (2+c)\log n) \leq n^{-h(c/2) + o(1)}$ for any constant $c > 0$, with an analogous lower tail bound for $0 < c < 2$, where $h(u) = (1+u)\ln(1+u) - u$ is the Bennett rate function.
Furthermore, we establish that this logarithmic scaling is robust: for arbitrary or uniformly random start--target pairs, the expected routing complexity remains $\E[S_{s,t}] = 2\log n + O(1)$, closely mirroring decentralized routing scenarios in real-world networks where endpoints are chosen dynamically rather than fixed a priori.
\end{abstract}

\section{Introduction}
\label{sec:introduction}
The existence of networks capable of efficiently transmitting and routing information was first highlighted by Stanley Milgram's famous 1967 experiment on small-world phenomena~\cite{Milgram67}. Several random graph models have since been proposed to explain this phenomenon. The most widely adopted, due to its rigorous theoretical foundation, is Kleinberg's small-world model~\cite{Kleinberg00}. Kleinberg's model has served as a basis for designing peer-to-peer and sensor networks that guarantee logarithmic routing, and its generalizations have been extensively studied in works such as~\cite{prokhorenkova2020graph}. Recent efforts by Prokhorenkova and Shekhovtsov~\cite{prokhorenkova2020graph} have further bridged the gap between practice and theory by analyzing the navigability of graph-based nearest neighbor search structures, though often under simplifying assumptions.

Despite the thorough analysis of Kleinberg's model, it does not explain \emph{why} or \emph{how} navigational effects emerge in natural networks. It is difficult to imagine that individuals in social networks form connections by computing connection probabilities inversely proportional to distance raised to a power $d$ (the ambient dimension). The model is inherently static and somewhat artificial.

In contrast, dynamic models based on network growth offer a more natural explanation. In 2011, Ponomarenko et al.~\cite{Ponomarenko2011} introduced a growth-based graph construction algorithm specifically designed for approximate nearest neighbor search, demonstrating that navigable properties arise as a direct consequence of the network's sequential formation. Their experimental study explicitly claimed that ``Simulation results confirm logarithmic dependency of search complexity from the number of elements in the structure.'' Subsequent work~\cite{malkov2012scalable, malkov2014approximate} systematically investigated the average hop count induced by greedy walks across Euclidean spaces of different dimensionalities, further reinforcing the empirical observation of polylogarithmic routing behavior. Together, these studies established an informal conjecture: greedy routing in sequentially grown $k$-nearest-neighbor graphs operates in polylogarithmic (and likely logarithmic) time. However, these conclusions were primarily based on numerical experiments and geometric intuition, lacking a fully rigorous mathematical proof.

Analyzing the general multidimensional case presents a significant technical challenge due to complex spatial dependencies and the absence of a natural linear ordering. As a foundational step toward this broader goal, we focus on rigorously solving the one-dimensional case. This setting captures the core combinatorial mechanisms of sequential nearest-neighbor attachment while remaining analytically tractable, and serves as the necessary logical bridge between empirical observations and a complete theory for higher-dimensional continuous spaces.

In this paper, we take this step by providing a rigorous analysis of greedy routing in a natural random graph model on the one-dimensional integer lattice. Vertices $\{1,\dots,n\}$ are embedded in $\mathbb{Z}$ and inserted sequentially in a uniformly random order. Upon insertion, each vertex connects to its nearest already-inserted neighbors on the left and right (if they exist). This construction yields a random geometric graph with non-local edges whose structure is governed by permutation statistics.

Let $\GW$ denote the greedy walk targeting vertex $n$, starting at vertex $1$. At each step from vertex $x$, $\GW$ moves to the neighbor $y \in N_G(x)$ that minimizes $|y - n|$ (i.e., the neighbor geometrically closest to the target). The walk terminates upon reaching $n$. We investigate the random variable $S_n$, denoting the number of steps required.

\medskip
\noindent\textbf{Main result.} The number of steps satisfies
\[
S_n = L_n + R_n - 2,
\]
where $L_n$ is the number of left-to-right minima and $R_n$ the number of right-to-left minima in a uniform random permutation of $\{1,\dots,n\}$. Consequently:
\begin{enumerate}[label=(\roman*)]
\item $\E[S_n] = 2H_n - 2 = 2\log n + O(1)$, where $H_n$ is the $n$-th harmonic number,
\item $\Var(S_n) = 2\log n + O(1)$,
\item For any constant $c > 0$, $\prob(S_n \geq (2+c)\log n) \leq n^{-h(c/2) + o(1)}$, and for $0 < c < 2$, $\prob(S_n \leq (2-c)\log n) \leq n^{-h(-c/2) + o(1)}$, where $h(u) = (1+u)\ln(1+u) - u$.
\end{enumerate}
Thus $S_n = \Theta(\log n)$ with high probability (w.h.p.), confirming and strengthening the conjecture that $S_n$ is polylogarithmic.

\medskip
\noindent\textbf{Technical contribution.} The key insight is a combinatorial characterization: the greedy path from $1$ to $n$ traverses exactly the left-to-right minima until reaching the global minimum position, then follows the right-to-left minima to $n$. This reduces the routing problem to classical permutation statistics, bypassing intricate edge-probability calculations. Our analysis leverages the independence of record indicators in random permutations to establish sharp concentration.

\medskip
\noindent\textbf{Related work.} 
Kleinberg's seminal small-world model~\cite{Kleinberg00} demonstrates that augmenting a $d$-dimensional lattice with long-range edges drawn from an inverse-power law (with decay exponent $\alpha=d$) enables greedy routing in $\Theta(\log^2 n)$ steps. Martel and Nguyen~\cite{Martel04} later established that this $\Theta(\log^2 n)$ bound is tight for all fixed dimensions $d \geq 1$, including the 1D case. 
Structured peer-to-peer overlay networks, such as Chord~\cite{Stoica2001}, Kademlia~\cite{Maymounkov2002}, and Tapestry~\cite{Zhao2004}, also achieve $O(\log n)$ routing complexity. However, these systems rely on engineered identifier spaces (e.g., consistent hashing or XOR metrics) and explicit maintenance of $O(\log n)$ routing table entries per node. 
Our model differs fundamentally: edges emerge from a sequential nearest-neighbor insertion process rather than a predefined spatial probability distribution or engineered overlay, yet it achieves strictly faster $\Theta(\log n)$ greedy routing without explicit routing tables. 
While our construction can be interpreted as a geometric variant of preferential attachment---where early-inserted vertices implicitly accumulate connections through spatial proximity rather than explicit degree-biased rules---it is important to contextualize this within the broader mathematical literature. The original Barab\'asi--Albert framework was highly influential as a conceptual pioneer, though its formal mathematical properties were later rigorously established in works such as Bollob\'as's survey on scale-free random graphs~\cite{BollobasHandbook} and the generalized attachment models of Ostroumova et al.~\cite{Ostroumova2013}. These mathematical developments primarily target \emph{structural} network characteristics, notably power-law degree distributions and tunable clustering coefficients. Crucially, they do not address \emph{navigability} or the efficiency of decentralized greedy routing. Our model shifts the analytical focus from structural emulation to algorithmic performance, demonstrating that a purely local, geometry-driven growth process implicitly generates the long-range edges required for $\Theta(\log n)$ routing, without any degree-proportional attachment mechanism.
Beyond the canonical $1 \to n$ routing, our analysis extends naturally to arbitrary source--target pairs. As shown in Corollary~\ref{cor:random-st}, when the start and target vertices are chosen uniformly at random, the expected routing complexity remains $\E[S_{s,t}] = 2\log n + O(1)$. This finding is particularly significant for modeling real-world decentralized networks, where endpoints are rarely fixed in advance but instead arise dynamically from user queries, peer interactions, or distributed lookup protocols. The persistence of logarithmic hop counts under random endpoint selection underscores the intrinsic navigability of sequentially grown nearest-neighbor graphs and reinforces the practical relevance of our theoretical guarantees.
The combinatorial reduction to permutation records connects our analysis directly to the well-established theory of \emph{Cartesian trees} and \emph{treaps}. Vuillemin~\cite{Vuillemin80} introduced Cartesian trees as a binary tree structure where an inorder traversal recovers the original sequence order, while the heap property is satisfied with respect to an associated priority value. When priorities are assigned via a uniform random permutation (equivalently, random insertion times), the resulting structure is a treap, as formalized by Aragon and Seidel~\cite{AragonSeidel89}. Our sequentially grown graph $G_n$ implicitly encodes precisely this Cartesian tree: the parent of any vertex $x$ in the tree corresponds to the first already-inserted neighbor of $x$ on either side, and the global minimum (inserted first) forms the root. The greedy routing path from $1$ to $n$ exactly traces the unique tree path from the leftmost leaf to the rightmost leaf through the root. While the structural and search-time properties of treaps and Cartesian trees have been thoroughly analyzed in the context of randomized search trees and algorithmic complexity~\cite{Devroye86, SeidelA96}, our contribution shifts the analytical lens to \emph{decentralized greedy routing} on the induced geometric graph. Rather than studying tree height or pointer-based search path length, we characterize the spatial hop count of a distance-minimizing greedy walk. This reframing bridges classical permutation statistics with modern geometric routing theory and provides a rigorous foundation for the empirically observed navigability of growth-based nearest-neighbor structures~\cite{malkov2016growing}.

\section{Model and Preliminaries}
\label{sec:model}

\subsection{Graph construction}

Let $V = \{1, 2, \dots, n\} \subset \mathbb{Z}$ be vertices embedded on the integer line. Let $\pi = (\pi(1), \pi(2), \dots, \pi(n))$ be a uniform random permutation of $V$, representing the insertion order. Define the \emph{insertion time function} $\tau: V \to \{1,\dots,n\}$ by $\tau(x) = t$ iff $\pi(t) = x$. Thus $\tau$ is also a uniform random permutation.

The graph $G_n = (V, E)$ is constructed incrementally:
\begin{itemize}
\item Initialize $G_n$ with vertex set $V$ and $E = \emptyset$.
\item For $t = 1$ to $n$:
  \begin{itemize}
  \item Insert vertex $x = \pi(t)$.
  \item If there exists $y < x$ already inserted (i.e., $\tau(y) < t$), connect $x$ to the \emph{closest} such $y$ (maximizing $y$ subject to $y < x$ and $\tau(y) < t$).
  \item If there exists $y > x$ already inserted, connect $x$ to the \emph{closest} such $y$ (minimizing $y$ subject to $y > x$ and $\tau(y) < t$).
  \end{itemize}
\end{itemize}
Edges are undirected. Note that adjacent vertices $i$ and $i+1$ always share an edge (vacuously, as no vertex lies strictly between them).

\subsection{Permutation records}

For a permutation $\sigma$ of $\{1,\dots,n\}$:
\begin{itemize}
\item A position $i$ is a \emph{left-to-right minimum} (\ltr minimum) if $\sigma(i) < \min\{\sigma(1), \dots, \sigma(i-1)\}$ (with $i=1$ always an \ltr minimum).
\item A position $i$ is a \emph{right-to-left minimum} (\rtl minimum) if $\sigma(i) < \min\{\sigma(i+1), \dots, \sigma(n)\}$ (with $i=n$ always an \rtl minimum).
\end{itemize}
Let $L_n(\sigma)$ and $R_n(\sigma)$ denote the counts of \ltr and \rtl minima, respectively. For uniform random $\sigma$:
\begin{equation}
\label{eq:record-indicators}
L_n = \sum_{i=1}^n I_i, \quad \text{where } I_i \sim \text{Bernoulli}(1/i) \text{ are independent}.
\end{equation}
Similarly for $R_n$ (by symmetry). Hence $\E[L_n] = H_n = \sum_{i=1}^n 1/i = \log n + \gamma + o(1)$ and $\Var(L_n) = \sum_{i=1}^n (1/i)(1-1/i) = \log n + O(1)$.

\section{Structural Characterization of Edges}
\label{sec:structure}

We first characterize when an edge $\{x,y\}$ with $x < y$ exists in $G_n$.

\begin{lemma}[Edge existence criterion]
\label{lem:edge-criterion}
For $1 \leq x < y \leq n$, the edge $\{x,y\}$ belongs to $E$ if and only if
\[
\max\{\tau(x), \tau(y)\} < \min_{z \in (x,y)} \tau(z),
\]
where $(x,y) = \{x+1, \dots, y-1\}$. Equivalently, $x$ and $y$ are the two vertices with smallest insertion times in the interval $[x,y]$.
\end{lemma}

\begin{proof}
Suppose $\tau(x) < \tau(y)$ (the case $\tau(y) < \tau(x)$ is symmetric). When $y$ is inserted at time $\tau(y)$, the already-inserted vertices in $[x,y]$ are precisely $\{z \in [x,y] : \tau(z) < \tau(y)\}$. Vertex $y$ connects to the closest such vertex on the left, which is $\max\{z < y : \tau(z) < \tau(y)\}$. This equals $x$ iff no vertex in $(x,y)$ has insertion time less than $\tau(y)$, i.e., $\tau(y) < \min_{z \in (x,y)} \tau(z)$. Since $\tau(x) < \tau(y)$ by assumption, this is equivalent to $\max\{\tau(x),\tau(y)\} = \tau(y) < \min_{z \in (x,y)} \tau(z)$.

Conversely, if the inequality holds, then at time $\max\{\tau(x),\tau(y)\}$, the later-inserted vertex connects to the earlier one as its nearest neighbor on the appropriate side, creating edge $\{x,y\}$.
\end{proof}

\begin{corollary}
\label{cor:adjacent-edges}
For all $1 \leq i < n$, edge $\{i, i+1\}$ exists in $G_n$ (since $(i,i+1) = \emptyset$).
\end{corollary}

Thus $G_n$ always contains the path $1-2-\cdots-n$, ensuring connectivity and that greedy walks never terminate prematurely.

Figure~\ref{fig:graph16} illustrates a random instance with $n = 16$. The axes place coordinate $x$ horizontally and insertion time $\tau(x)$ vertically, so the edge criterion of Lemma~\ref{lem:edge-criterion} has a direct geometric reading: edge $\{x,y\}$ exists iff no vertex between $x$ and $y$ lies \emph{above} both endpoints in the plot (i.e., has a smaller $\tau$-value).

\begin{figure}[t]
\centering
\includegraphics[width=1.0\textwidth]{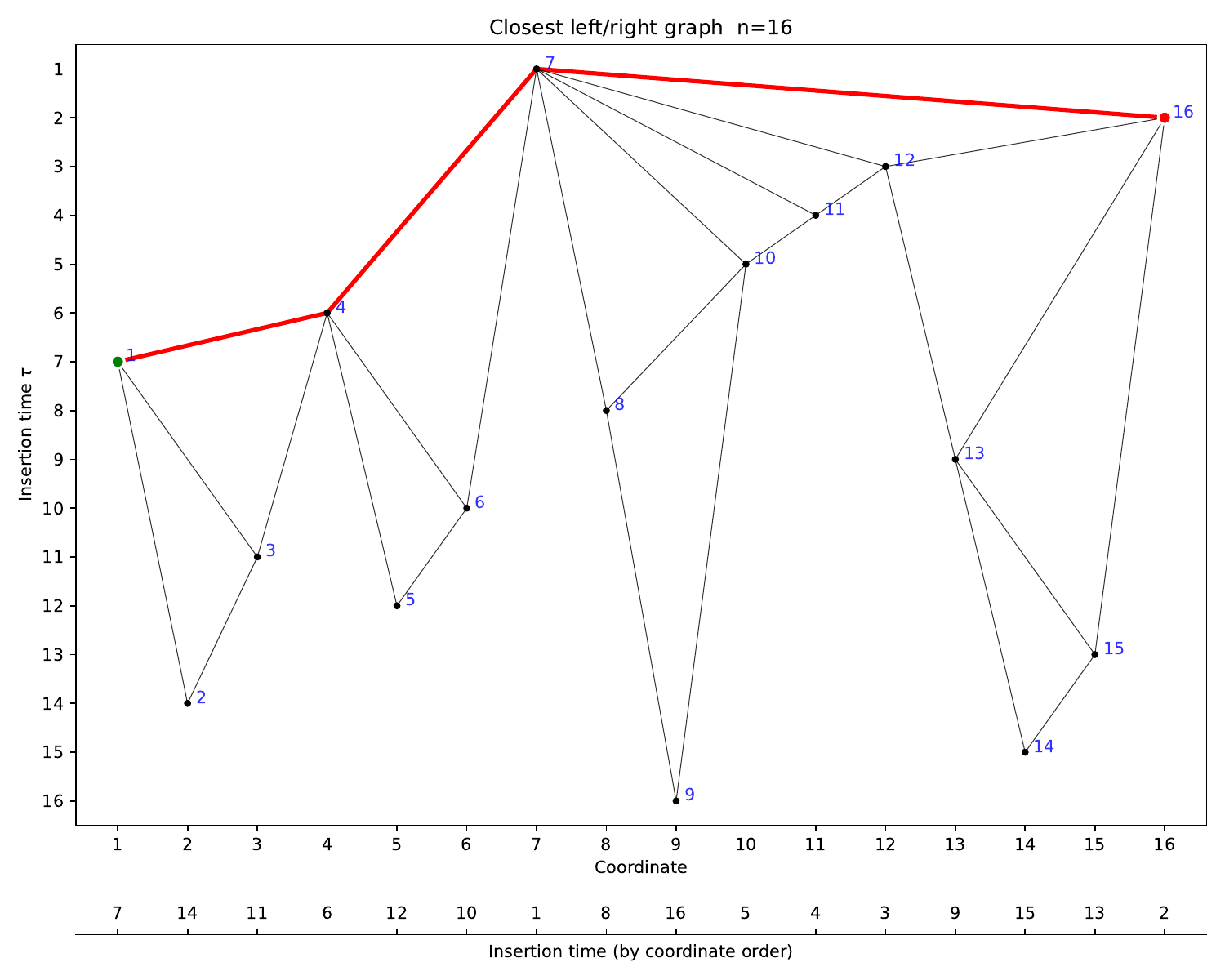}
\caption{The graph $G_{16}$ for the permutation $\tau = (7,14,11,6,12,10,1,8,16,5,4,3,9,15,13,2)$
(insertion times listed by vertex $1,2,\dots,16$). Vertex~$7$ is inserted first ($\tau(7)=1$) and
is the global minimum. The \ltr minima of $\tau$ are $\{1,4,7\}$ (vertices with
$\tau$-values $7>6>1$, each setting a new record scanning left to right), and the \rtl
minima are $\{7,16\}$. By Theorem~\ref{thm:greedy-path}, the greedy walk from $1$ to $16$ follows
$1 \to 4 \to 7 \to 16$, taking $S_{16} = 3 + 2 - 2 = 3$ steps.}
\label{fig:graph16}
\end{figure}

\section{Greedy Walk Dynamics}
\label{sec:greedy-walk}

Fix target $q = n$ and start vertex $x_0 = 1$. At current vertex $x < n$, the greedy walk moves to
\[
\GW(x) = \arg\min_{y \in N_G(x)} |y - n| = \arg\max_{y \in N_G(x) \cap (x,n]} y,
\]
since $|y-n| = n-y$ is decreasing in $y$. Thus $\GW$ always moves rightward to the rightmost neighbor.

\subsection{Neighborhood structure}

For $x \in V$, define:
\begin{align*}
r(x) &= \min\{y > x : \tau(y) < \tau(x)\} \quad \text{(next smaller insertion time to the right)}, \\
\ell(x) &= \max\{y < x : \tau(y) < \tau(x)\} \quad \text{(previous smaller insertion time to the left)}.
\end{align*}
If no such $y$ exists, the value is undefined. Let $m = \arg\min_{z \in V} \tau(z)$ be the position of the global minimum (inserted first). We call $y \in (x,n]$ a \emph{right-record of $x$} if $\tau(y) < \min_{z \in (x,y)} \tau(z)$, i.e., $y$ is an \ltr minimum when scanning $(x,n]$ from left to right.

\begin{lemma}[Rightmost LTR minimum is the earliest-inserted suffix vertex]
\label{lem:rightmost-ltr-is-argmin}
For any $x < n$, if $\tau(x) = \min_{z \in [x,n]} \tau(z)$ (i.e., $r(x)$ is undefined), then the rightmost right-record of $x$ is
\[
\arg\min_{y > x} \tau(y).
\]
\end{lemma}

\begin{proof}
Since $\tau(x)$ is the global minimum on $[x,n]$, every $y \in (x,n]$ satisfies $\tau(y) > \tau(x)$. Recall that $y \in (x,n]$ is an right-record of $x$ iff $\tau(y) < \min_{z \in (x,y)} \tau(z)$. Let $y^* = \arg\min_{y > x} \tau(y)$. For any $y > y^*$, we have $\tau(y^*) < \tau(y)$ and $y^* \in (x, y)$, so $y$ fails this condition and is not a right-record of $x$. Conversely, $y^*$ itself satisfies the condition since it achieves the minimum over all of $(x,n]$. Hence $y^*$ is the rightmost right-record of $x$.
\end{proof}

\begin{lemma}[Right-neighbor characterization]
\label{lem:right-neighbors}
For $x < n$, the right-neighbors of $x$ are precisely the right-records of $x$ that lie in $[x+1, r(x)]$ if $r(x)$ exists, or all right-records of $x$ if $r(x)$ is undefined. Moreover, the largest right-neighbor is:
\[
\max \bigl(N_G(x) \cap (x,n]\bigr) =
\begin{cases}
r(x) & \text{if } r(x) \text{ exists}, \\
\arg\min_{y > x} \tau(y) & \text{otherwise}.
\end{cases}
\]
In the second case, $\arg\min_{y>x}\tau(y)$ is the rightmost right-neighbor because it is the rightmost right-record of $x$ (Lemma~\ref{lem:rightmost-ltr-is-argmin}).
\end{lemma}

\begin{proof}
By Lemma~\ref{lem:edge-criterion}, $y > x$ is a neighbor iff $\max\{\tau(x),\tau(y)\} < \min_{z \in (x,y)} \tau(z)$. If $\tau(y) < \tau(x)$, this requires $\tau(x) < \min_{z \in (x,y)} \tau(z)$, which holds iff $y \leq r(x)$ (if $r(x)$ exists). If $\tau(y) > \tau(x)$, it requires $\tau(y) < \min_{z \in (x,y)} \tau(z)$, meaning $y$ is a right-record of $x$. However, for $y > r(x)$ (when $r(x)$ exists), we have $\tau(r(x)) < \tau(x) < \tau(y)$ and $r(x) \in (x,y)$, violating the condition. Thus only $y \leq r(x)$ can be neighbors when $r(x)$ exists. Among these, $r(x)$ itself satisfies the edge condition (as $\tau(r(x)) < \tau(x)$ and no $z \in (x,r(x))$ has $\tau(z) < \tau(x)$ by definition of $r(x)$), and is the rightmost neighbor. When $r(x)$ is undefined, $\tau(x)$ is minimal in $[x,n]$, so neighbors are exactly the right-records of $x$; the rightmost such neighbor is $\arg\min_{y > x} \tau(y)$ by Lemma~\ref{lem:rightmost-ltr-is-argmin}.
\end{proof}

\subsection{Greedy path characterization}

\begin{theorem}[Greedy path = record sequence]
\label{thm:greedy-path}
The greedy walk from $1$ to $n$ traverses precisely:
\begin{enumerate}[label=(\alph*)]
\item All \ltr minima in order from $1$ to $m$ (the global minimum position),
\item Followed by all \rtl minima strictly greater than $m$ in order to $n$.
\end{enumerate}
Consequently, the number of steps is
\[
S_n = L_n + R_n - 2,
\]
where $L_n = L_n(\tau)$ and $R_n = R_n(\tau)$.
\end{theorem}

\begin{proof}
Let $x_0 = 1, x_1, x_2, \dots, x_k = n$ be the greedy path. Since $x_0 = 1$ is an \ltr minimum, we proceed by induction.

\textbf{Phase 1 (until $m$):} Suppose $x_i$ is an \ltr minimum with $x_i < m$. Then $\tau(x_i) > \tau(m)$ (as $m$ has the global minimum), so $r(x_i)$ exists and $r(x_i) \leq m$. By Lemma~\ref{lem:right-neighbors}, $x_{i+1} = r(x_i)$. But $r(x_i)$ is precisely the next position $y > x_i$ with $\tau(y) < \tau(x_i)$, which by definition is the next \ltr minimum. This sequence continues until reaching $m$, the rightmost \ltr minimum (since $\tau(m) = 1$ is minimal).

\textbf{Phase 2 (after $m$):} At $x_i = m$, no $y > m$ satisfies $\tau(y) < \tau(m)$, so $r(m)$ is undefined. By Lemma~\ref{lem:right-neighbors}, $x_{i+1} = \arg\min_{y > m} \tau(y)$. This vertex is an \rtl minimum (as it has minimal $\tau$ in $[x_{i+1}, n]$). Inductively, if $x_j > m$ is an \rtl minimum, then $\tau(x_j) < \min_{z > x_j} \tau(z)$ implies no vertex in $(x_j, n]$ has smaller insertion time, so the largest neighbor is $\arg\min_{y > x_j} \tau(y)$, which is the next \rtl minimum. This continues until $n$, which is always an \rtl minimum.

\textbf{Step count:} Phase 1 visits all $L_n$ \ltr minima, requiring $L_n - 1$ steps. Phase 2 visits all \rtl minima strictly greater than $m$; since $m$ is itself an \rtl minimum, there are $R_n - 1$ such vertices, requiring $R_n - 1$ steps. Total steps: $(L_n - 1) + (R_n - 1) = L_n + R_n - 2$.
\end{proof}

\begin{remark}
The formula $S_n = L_n + R_n - 2$ is \emph{exact} for every permutation $\tau$, not just in expectation. Figure~\ref{fig:graph16} illustrates the construction and greedy path for a small instance.
\end{remark}

The same combinatorial structure governs greedy walks between arbitrary vertices.

\begin{corollary}[Arbitrary start and target]
\label{cor:arbitrary-st}
For any $1 \leq s < t \leq n$, let $S_{s,t}$ denote the number of steps in the greedy walk from $s$ to $t$ in $G_n$. Then
\[
S_{s,t} = L(s,t) + R(s,t) - 2,
\]
where $L(s,t)$ and $R(s,t)$ are the numbers of \ltr and \rtl minima of $\tau$ restricted to $[s,t]$.
\end{corollary}

\begin{proof}
The greedy walk from $s$ targeting $t$ moves rightward at each step, visiting only vertices in $[s,t]$. By Lemma~\ref{lem:edge-criterion}, the edge $\{x,y\}$ with $s \leq x < y \leq t$ depends only on the relative ordering of insertion times in $[x,y]$. Since the graph structure and greedy dynamics within $[s,t]$ depend only on $\tau|_{[s,t]}$, Theorem~\ref{thm:greedy-path} applies with $[s,t]$ in place of $[1,n]$.
\end{proof}

\section{Probabilistic Analysis}
\label{sec:prob-analysis}

Since $\tau$ is a uniform random permutation, $L_n$ and $R_n$ are identically distributed (though not independent). We analyze $S_n = L_n + R_n - 2$.

\begin{theorem}[Expectation and variance]
\label{thm:expectation}
\[
\E[S_n] = 2H_n - 2 = 2\log n + 2\gamma - 2 + o(1),
\]
\[
\Var(S_n) = 2\left(H_n - \sum_{i=1}^n \frac{1}{i^2}\right) + 2\Cov(L_n, R_n) = 2\log n + O(1),
\]
where $H_n = \sum_{i=1}^n 1/i$ and $\gamma \approx 0.57721$ is Euler's constant.
\end{theorem}

\begin{proof}
By linearity and symmetry, $\E[S_n] = \E[L_n] + \E[R_n] - 2 = 2H_n - 2$. For variance:
\[
\Var(S_n) = \Var(L_n) + \Var(R_n) + 2\Cov(L_n, R_n).
\]
We have $\Var(L_n) = \sum_{i=1}^n \frac{1}{i}\left(1 - \frac{1}{i}\right) = H_n - \sum_{i=1}^n \frac{1}{i^2}$. Similarly for $\Var(R_n)$. The covariance satisfies $\Cov(L_n, R_n) = O(1)$ and is strictly positive (a known result for record statistics in random permutations), so:
\[
\Var(S_n) = 2\left(H_n - \sum_{i=1}^n \frac{1}{i^2}\right) + 2\Cov(L_n, R_n) = 2\log n + O(1).
\]
\end{proof}

\begin{remark}[Typical-case behavior and Gaussian approximation]
\label{rem:gaussian-approx}
While Theorem~\ref{thm:concentration-sharp} provides non-asymptotic tail bounds, the exact identity $S_n = L_n + R_n - 2$ combined with the classical Central Limit Theorem for record statistics~\cite{Gnedin08} implies a stronger distributional limit:
\begin{equation}
\label{eq:clt-sn}
\frac{S_n - (2H_n - 2)}{\sqrt{2\log n}} \xrightarrow{d} \mathcal{N}(0,1) \quad \text{as } n \to \infty.
\end{equation}
Figure~\ref{fig:sn-distribution} illustrates this convergence: for increasing $n$, the empirical distribution of $S_n$ becomes increasingly Gaussian, centered at $2\log n$ with standard deviation $\sqrt{2\log n}$. This explains why the Bennett-type bounds of Theorem~\ref{thm:concentration-sharp} are qualitatively sharp---the rate function $h(u)$ matches the Gaussian tail $\exp(-u^2/2)$ for small $u$, while providing valid (though conservative) bounds for large deviations where the union bound dominates.
\end{remark}

\begin{figure}[t]
\centering
\includegraphics[width=\textwidth]{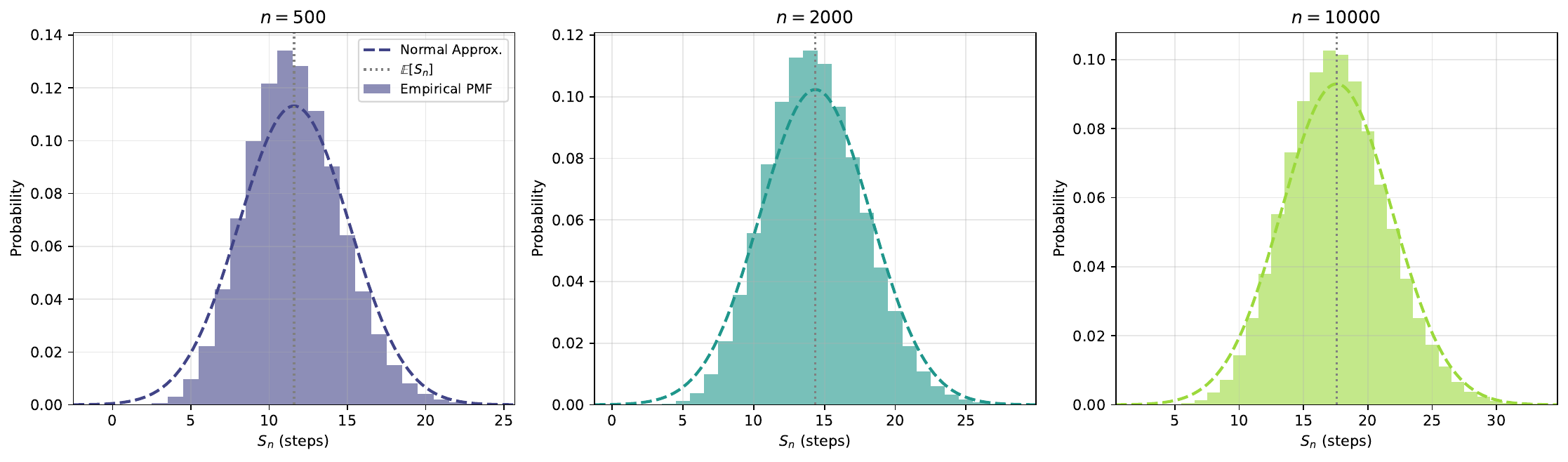}
\caption{Empirical probability mass function of $S_n$ for $n \in \{500, 2000, 10000\}$ (blue bars) compared to the asymptotic normal approximation $\mathcal{N}(2H_n-2,\, 2\log n)$ (dashed red). The dotted line marks $\E[S_n]$.}
\label{fig:sn-distribution}
\end{figure}

\begin{theorem}[Sharp concentration (explicit bounds)]
\label{thm:concentration-sharp}
Let $h(u) = (1+u)\ln(1+u) - u$ for $u > -1$, and $h(u)=\infty$ otherwise.
For any $n\ge 2$ and any $t>0$,
\begin{align*}
\prob\!\bigl(S_n \ge 2H_n + t - 2\bigr) &\le 2\exp\!\Bigl( - \sigma_n^2 \, h\!\Bigl(\frac{t}{2\sigma_n^2}\Bigr) \Bigr), \\[4pt]
\prob\!\bigl(S_n \le 2H_n - t - 2\bigr) &\le 2\exp\!\Bigl( - \sigma_n^2 \, h\!\Bigl(-\frac{t}{2\sigma_n^2}\Bigr) \Bigr),
\end{align*}
where $\sigma_n^2 = \Var(L_n) = H_n - \sum_{i=1}^n \frac{1}{i^2}$.
In asymptotic form, for any constant $c>0$,
\[
\prob\bigl(S_n \ge (2+c)\log n\bigr) \le n^{-h(c/2) + o(1)},
\]
and for $0<c<2$,
\[
\prob\bigl(S_n \le (2-c)\log n\bigr) \le n^{-h(-c/2) + o(1)}.
\]
\end{theorem}

\begin{proof}
Recall that $L_n = \sum_{i=1}^n I_i$ with independent $I_i \sim \text{Bernoulli}(1/i)$.
Hence $L_n$ is a sum of independent $[0,1]$-bounded variables with mean $\mu_n = \E[L_n] = H_n$ 
and variance $\sigma_n^2 = \Var(L_n) = H_n - \sum_{i=1}^n i^{-2}$. By symmetry, $R_n$ has the same mean and variance.

\paragraph{Upper tail.}
Bennett's inequality~\cite{Bennett62} asserts that for independent $X_i \in [0,1]$ with $\sum \E[X_i] = \mu$ and $\sum \Var(X_i) = v$,
\[
\prob\!\Bigl(\sum X_i \ge \mu + s\Bigr) \le \exp\!\Bigl( -v\, h\!\Bigl(\frac{s}{v}\Bigr) \Bigr), \quad s>0,
\]
where $h(u) = (1+u)\ln(1+u)-u$. Since $S_n = L_n + R_n - 2$ and $\E[S_n] = 2\mu_n - 2$, the event $\{S_n \ge 2\mu_n + t - 2\}$ is contained in $\{L_n \ge \mu_n + t/2\} \cup \{R_n \ge \mu_n + t/2\}$. Applying Bennett's inequality with $s = t/2$ to each and taking the union bound:
\[
\prob(S_n \ge 2\mu_n + t - 2) \le 2\exp\!\Bigl( -\sigma_n^2\, h\!\Bigl(\frac{t}{2\sigma_n^2}\Bigr) \Bigr),
\]
which is the first finite-$n$ bound.

\paragraph{Lower tail.}
For sums of non-negative variables, the left-tail analogue holds: for $s \in (0,\mu)$,
\[
\prob\!\Bigl(\sum X_i \le \mu - s\Bigr) \le \exp\!\Bigl( -v\, h\!\Bigl(-\frac{s}{v}\Bigr) \Bigr).
\]
(This follows by the Cramér--Chernoff method, yielding the same rate function $h$ as the Legendre--Fenchel transform of the cumulant generating function.)
Similarly to the upper tail, the event $\{S_n \le 2\mu_n - t - 2\}$ is contained in $\{L_n \le \mu_n - t/2\} \cup \{R_n \le \mu_n - t/2\}$. The union bound gives
\[
\prob(S_n \le 2\mu_n - t - 2) \le 2\exp\!\Bigl( -\sigma_n^2\, h\!\Bigl(-\frac{t}{2\sigma_n^2}\Bigr) \Bigr).
\]

\paragraph{Asymptotic form.}
We have $H_n = \log n + \gamma + o(1)$ and $\sigma_n^2 = \log n + O(1)$.
For the upper tail, set $t_{\mathrm{tot}} = 2\mu_n + t_{\mathrm{tot}} - 2 = (2+c)\log n$, i.e.\ $t_{\mathrm{tot}} = c\log n + 2 - 2\gamma + o(1)$. Then
\[
\frac{t_{\mathrm{tot}}}{2\sigma_n^2} = \frac{c\log n + O(1)}{2(\log n + O(1))} \to \frac{c}{2}.
\]
By continuity of $h$,
\[
\prob\bigl(S_n \ge (2+c)\log n\bigr) \le 2\exp\!\Bigl( -(\log n + O(1))\, h\!\Bigl(\frac{c}{2}+o(1)\Bigr) \Bigr) = n^{-h(c/2) + o(1)},
\]
where the factor $2 = e^{\log 2}$ is absorbed into the $o(1)$ in the exponent.
For the lower tail, set $2\mu_n - t_{\mathrm{tot}} - 2 = (2-c)\log n$ with $0<c<2$, giving $t_{\mathrm{tot}} = c\log n + O(1)$. Then $t_{\mathrm{tot}}/(2\sigma_n^2) \to c/2 \in (0,1)$, so $-t_{\mathrm{tot}}/(2\sigma_n^2) \in (-1,0)$ for large $n$, ensuring $h(-t_{\mathrm{tot}}/(2\sigma_n^2))$ is well-defined. The identical calculation gives
\[
\prob\bigl(S_n \le (2-c)\log n\bigr) \le n^{-h(-c/2) + o(1)}.
\]
This completes the proof.
\end{proof}

\begin{remark}[Conservatism of the Union Bound and Numerical Verification]
\label{rem:conservatism}
Because $L_n$ and $R_n$ are positively correlated, $\Var(S_n) = 2\log n + O(1)$ exceeds the independent case, making large deviations of the sum more probable. Nevertheless, the union bound used in Theorem~\ref{thm:concentration-sharp} sums marginal tail probabilities without accounting for joint occurrences, so it remains a strict overestimate. The widening gap for larger $c$ is characteristic of large-deviation union bounds, where the marginal decay rate dominates the bound while the true joint tail probability decays more slowly.
\end{remark}

\begin{corollary}
\label{cor:whp}
For any $\varepsilon > 0$, $\prob\bigl((2-\varepsilon)\log n \leq S_n \leq (2+\varepsilon)\log n\bigr) \to 1$ as $n \to \infty$.
\end{corollary}

\begin{figure}[t]
    \centering
    \includegraphics[width=0.8\textwidth]{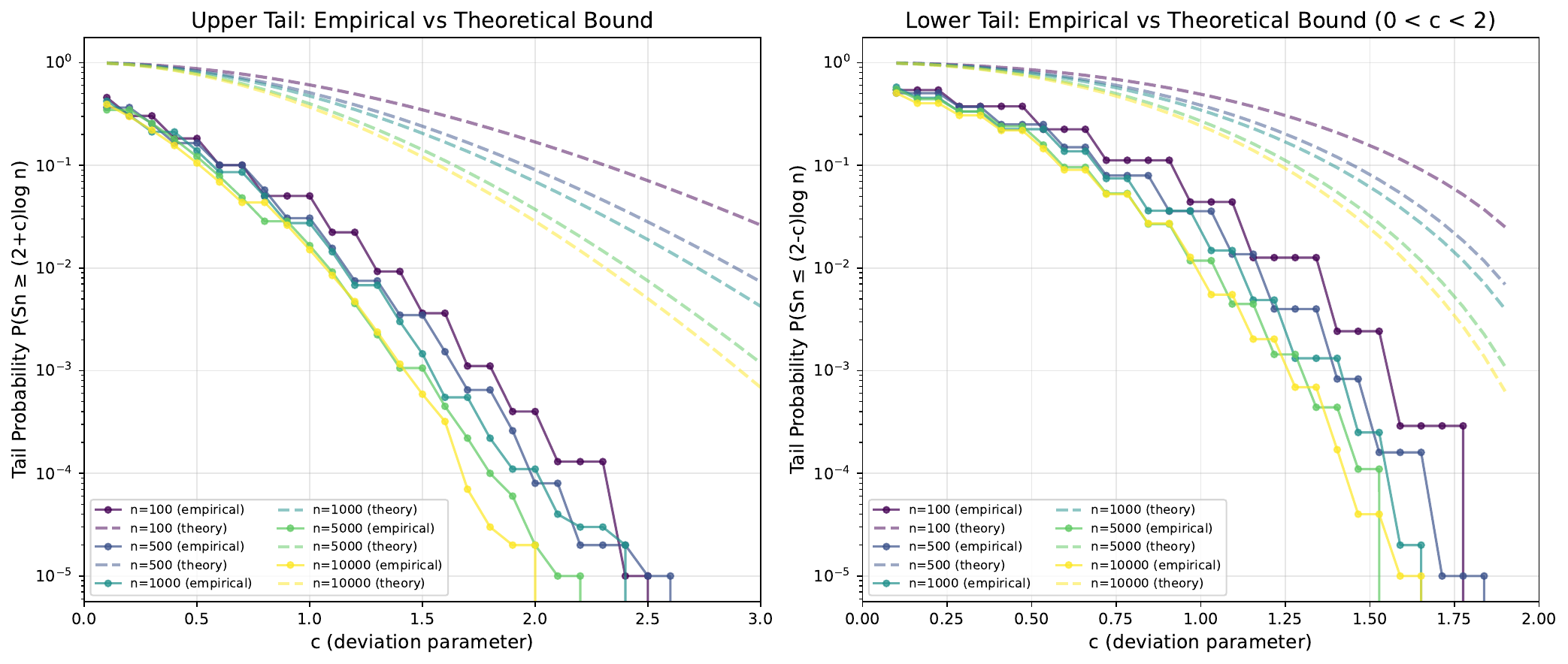}
    \caption{Empirical vs. theoretical tail bounds for $S_n$. \textbf{Left:} Upper tail $\prob(S_n \ge (2+c)\log n)$. \textbf{Right:} Lower tail $\prob(S_n \le (2-c)\log n)$. Dashed lines denote the Bennett-based bounds from Theorem~\ref{thm:concentration-sharp}; solid lines show empirical probabilities. The theoretical bounds are valid but conservative, as the union bound sums marginal tails without accounting for joint behavior or the positive correlation between $L_n$ and $R_n$.}
    \label{fig:tail_bounds}
\end{figure}

\subsection{Arbitrary start--target pairs}

Corollary~\ref{cor:arbitrary-st} generalizes the step count to arbitrary start--target pairs; its probabilistic consequences follow immediately.

\begin{corollary}[Expectation and concentration for arbitrary pairs]
\label{cor:prob-arbitrary-st}
For any $1 \leq s < t \leq n$,
\[
\E[S_{s,t}] = 2H_{t-s+1} - 2 \sim 2\log(t-s+1),
\]
and $S_{s,t} = \Theta(\log(t-s+1))$ with high probability.
\end{corollary}

\begin{proof}
Since $\tau|_{[s,t]}$ is a uniform random permutation of $t - s + 1$ elements (by exchangeability of uniform permutations), Theorems~\ref{thm:expectation} and~\ref{thm:concentration-sharp} apply to $S_{s,t} = L(s,t) + R(s,t) - 2$ (Corollary~\ref{cor:arbitrary-st}) with $n$ replaced by $t - s + 1$.
\end{proof}

\begin{corollary}[Uniformly random start and target]
\label{cor:random-st}
Let $s$ and $t$ be chosen uniformly at random from $\{1,\dots,n\}$ with $s \neq t$, independently of $\tau$. Then
\[
\E[S_{s,t}] = 2\log n + O(1).
\]
\end{corollary}
\begin{proof}
By symmetry, we condition on $s < t$. Let $D = t - s$ denote the distance between the endpoints. For a fixed $d \in \{1,\dots,n-1\}$, there are exactly $n-d$ ordered pairs $(s,t)$ with $s < t$ and $t-s = d$. Since there are $\binom{n}{2}$ such pairs in total, the conditional distribution of $D$ is triangular:
\[
\prob(D = d \mid s < t) = \frac{n-d}{\binom{n}{2}} = \frac{2(n-d)}{n(n-1)}.
\]
From Corollary~\ref{cor:prob-arbitrary-st}, we have $\E[S_{s,t} \mid D=d] = 2H_{d+1} - 2$. Taking the expectation over $D$ gives
\[
\E[S_{s,t}] = \sum_{d=1}^{n-1} \frac{2(n-d)}{n(n-1)} \bigl(2H_{d+1} - 2\bigr) = \frac{4}{n(n-1)} \sum_{d=1}^{n-1} (n-d) H_{d+1} - 2.
\]
To estimate the weighted sum, note that $H_{d+1} = \log d + \gamma + O(1/d)$. Using standard estimates for harmonic sums (or approximating by the integral $\int_1^n (n-x)\log x\,dx$), we obtain
\[
\sum_{d=1}^{n-1} (n-d) H_{d+1} = \frac{n^2}{2} \log n + O(n^2).
\]
Substituting this asymptotic into the expectation formula yields
\[
\E[S_{s,t}] = \frac{4}{n^2} \left( \frac{n^2}{2} \log n + O(n^2) \right) - 2 = 2\log n + O(1),
\]
which completes the proof.
\end{proof}

\begin{figure}[t]
\centering
\includegraphics[width=0.8\textwidth]{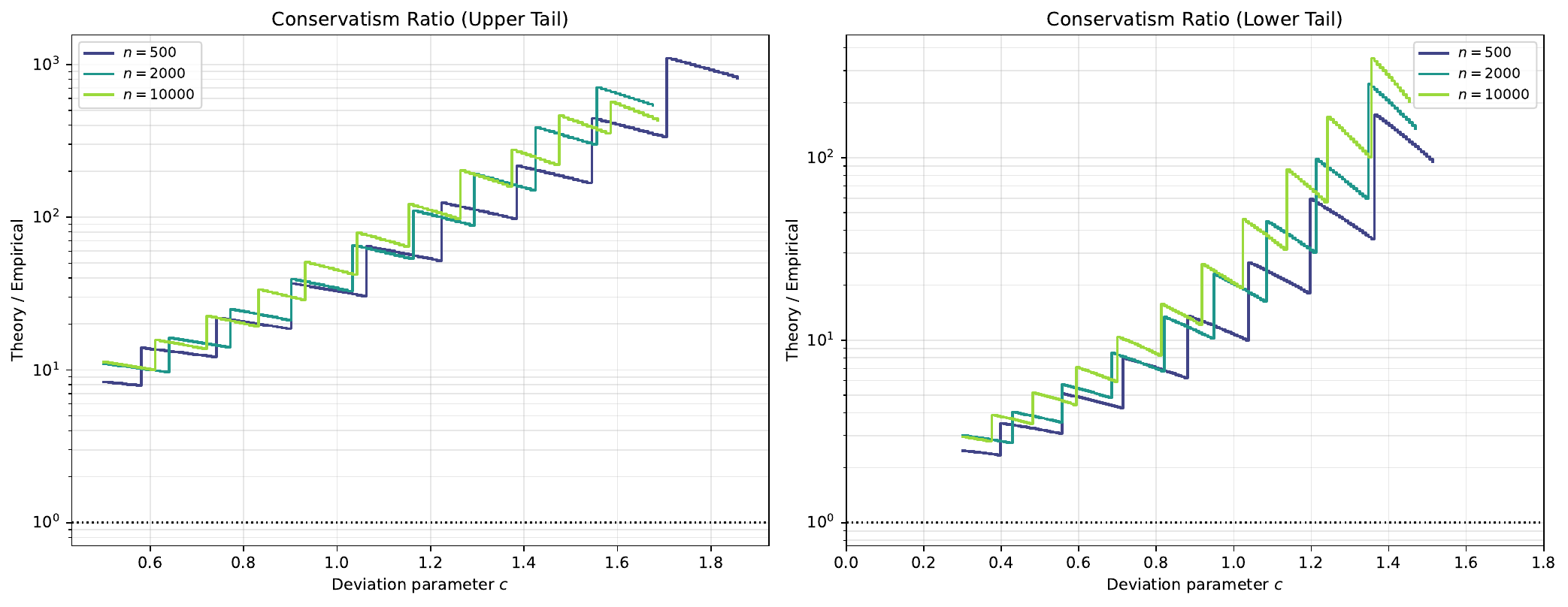}
\caption{Ratio of theoretical bound to empirical tail probability for the upper tail $P(S_n \ge (2+c)\log n)$. Values $>1$ indicate conservatism. The increasing trend with $c$ reflects the growing impact of the positive correlation between $L_n$ and $R_n$ in extreme tails. Curves terminate where empirical probability $< 10^{-6}$ (insufficient samples for reliable estimation).}
\label{fig:conservatism}
\end{figure}

\section{Conclusion and Open Problems}
\label{sec:conclusion}
We have established that greedy routing from vertex $1$ to $n$ in the sequentially grown 1D random graph requires exactly $S_n = L_n + R_n - 2$ steps, where $L_n$ and $R_n$ count the left-to-right and right-to-left minima in the insertion-time permutation. This exact combinatorial identity resolves the long-standing logarithmic routing conjecture that originated from empirical studies~\cite{Ponomarenko2011, malkov2016overlay}, demonstrating that routing complexity is tightly concentrated around $2\log n$. More fundamentally, it reveals how navigability emerges naturally from sequential geometric attachment, bypassing the need for engineered long-range edges or explicit degree-biased rules.

This analysis provides a foundational stepping stone toward understanding sequential growth in continuous and multidimensional spaces. We identify three primary directions for future research:
\begin{enumerate}
\item \textbf{Structural properties:} Characterize the degree distribution and graph diameter of the sequentially grown graph. While routing is logarithmic, the maximum degree and overall diameter depend on the accumulation of edges at record-breaking vertices and warrant precise asymptotic analysis.
\item \textbf{Directionless $K$-nearest neighbors in continuous space:} Analyze the variant where $n$ points are drawn i.i.d. uniformly from $[0,1]$ or $\mathbb{S}^1$, connecting each new point to its $K$ closest already-inserted neighbors based purely on Euclidean distance. The removal of left/right directional information introduces complex local dependencies.
\item \textbf{Higher-dimensional continuous domains:} Extend the sequential insertion model to the unit hypercube $[0,1]^d$ or $\mathbb{S}^d$ with uniform random points, where each new vertex connects to its $K$ nearest already-inserted neighbors.
\end{enumerate}

Motivated by these directions and empirical evidence from growth-based ANN structures~\cite{malkov2012scalable, malkov2014approximate, malkov2016growing}, we formalize the following conjecture for the continuous setting:
\begin{conjecture}
\label{conj:continuous-knn}
Let $n$ points be drawn i.i.d. uniformly from the unit circle $\mathbb{S}^1$ (or the interval $[0,1]$ with periodic boundary conditions). Construct a graph by sequentially inserting the points and connecting each new point to its $K$ nearest already-inserted neighbors. For every $\varepsilon > 0$, there exists a constant $K_0 = K_0(\varepsilon)$ such that for all $K \geq K_0$ and sufficiently large $n$, if $s$ and $t$ are chosen uniformly at random from the $n$ points, then with probability at least $1 - \varepsilon$:
\begin{enumerate}[label=(\roman*)]
\item the greedy walk from $s$ successfully reaches $t$, and
\item the walk completes in at most $C(\varepsilon) \log n$ steps, where $C(\varepsilon)$ is a constant depending only on $\varepsilon$.
\end{enumerate}
\end{conjecture}

We expect Conjecture~\ref{conj:continuous-knn} to extend to arbitrary dimensions $d \geq 1$, preserving logarithmic routing with a dimension-dependent constant. The principal technical barrier lies in the geometric structure of continuously grown graphs: without a natural linear ordering, the elegant permutation-record characterization breaks down, requiring new probabilistic and geometric tools to control local dependencies and long-range jump distributions.

From a network science perspective, our work complements the rigorous mathematical theory of preferential attachment~\cite{BollobasHandbook, Ostroumova2013}. While classical preferential attachment models successfully reproduce structural features like heavy-tailed degree distributions and clustering, they do not address decentralized navigability. By contrast, our sequentially grown nearest-neighbor graph demonstrates that algorithmic navigability can arise from purely local, distance-based growth dynamics. This highlights a fundamental divergence between structural realism and routing efficiency in network generation paradigms. The bridge we establish between geometric graph growth and permutation statistics may inform the analysis of dynamically evolving networks, the theoretical foundations of approximate nearest neighbor search, and models of spatial biological networks.

\section*{Acknowledgments}
The author thanks Valery Kalyagin for his inspiration and steadfast support.
Special gratitude is extended to Vladimir Krylov for introducing the author to this research area and for his ongoing guidance, and to Yury Malkov for our earlier collaborative work in this domain and for being the first to conceptualize the hierarchical structure that emerges in such graphs.
Much of this collaborative work took place at Meralabs, an environment that itself emerged largely due to the vision and efforts of the late Dmitry Ponomarev.
The author wishes to honor Dmitry's memory for his enthusiasm and shared ideas, which provided the initial spark and foundational direction for this work.
Additionally, the author is grateful to Andrey Sladkov for insightful discussions; in particular, his suggestion to revisit the simplest case proved instrumental in shaping the analytical approach of this paper.

The article was prepared within the framework of the Basic Research Program at the National Research University Higher School of Economics (HSE).

\bibliographystyle{plain}  
\bibliography{references}

\end{document}